\documentclass{amsart}
\usepackage{amsfonts,amssymb,amsmath,amsthm}
\usepackage{url}
\usepackage{enumerate}

\newtheorem{prop}{Proposition}[section]
\newtheorem{df}[prop]{Definition}
\newtheorem{thm}[prop]{Theorem}
\newtheorem{rema}[prop] {Remark}

\author{Enrico Boasso}


\begin{document}

\title[Drazin spectrum]{The Drazin Spectrum in  Banach Algebras}

\begin{abstract}Several basic properties of the Drazin spectrum in Banach algebras will be studied.  As an application, some results on meromorphic Banach space operators will be obtained.
\end{abstract}
\maketitle

\section{Introduction} \label{sect1}
\indent The Drazin spectrum of Banach space operators is closely related to other several spectra
such as the ascent, the descent and  the left and the right Drazin spectra, see for example \cite{K, B}. 
In \cite{B} the relationships
among all the aforementioned  spectra and the multiplication operators were studied both for Banach space bounded and linear maps
and for Banach algebra elements; what is more, a characterization of the spectrum in terms of the Drazin spectrum 
and the poles of the resolvent was also presented. On the other hand, in two recent articles the descent and the left Drazin spectra
of Banach space operators
were considered, see \cite{BKMO, BBO}. The main objective of this work is to study several basic properties of the Drazin spectrum
in Banach algebras similar to the ones presented in \cite{BKMO, BBO}.\par

\indent  In fact, in section 2 conditions equivalent to the emptyness of the Drazin spectrum will be given. 
In addition, it will be characterized when the Drazin spectrum is countable.
Moreover, if $a$ and $b$ belong to a 
Banach algebra $A$, then the Drazin spectra of $ab$ and $ba$ will be proved to coincide. On the other hand, in \cite[page 265]{BKMO}  was
pointed out that there is no relation between the descent spectra of $R_T$ and $T$, where $T\in L(H)$ is an operator
defined on the Hilbert space $H$, and $R_T\in L(L(H))$ is the right multiplication operator. In the next section
a description of the ascent and the descent spectra of the multiplication operators in Hilbert space will be given. To this end,
the relationships among the involution of a $C^*$-algebra and the ascent, the descent, and the Drazin spectra of the left and the right multipliation operators will
be characterized, see also \cite{B} for more results on the aforementioned spectra of the multiplication operators. 
Finally, as an application, in section 3 the Drazin spectrum will be used to prove several results on mermorphic Banach space operators.\par 
  
\indent The results of this work are related to the ones
presented at the 23th International Conference on Operator Theory held in
Timisoara, Romania, from June 29 to July 4, 2010. It is a
pleasure for the author to acknowledge the stimulating atmosphere
and the warm hospitality of the Conference. Moreover, the author
wish to express his fully indebtedness to the organizers of the
Conference, specially to Professor Dan Timotin, for the help to
attend the Conference.\par

\vskip.2truecm

\indent Before going on, several definitions and some notation will be recalled.\par

\indent From now on, $A$ will denote a unital Banach algebra and $e\in A$ will stand for the unit element of $A$.
If $a\in A$, then $L_a \colon A\to A$
and $R_a\colon A\to A$ will denote the maps defined by
left and right multiplication respectively, namely $L_a(x)=ax$ and $R_a(x)=xa$,   
where $x\in A$. \par

\indent Next follows the key notions of the present work.
 Given a Banach algebra $A$, an element $a\in A$ is said to be
\it Drazin invertible\rm, if there exist a necessarily unique $b\in A$ and some $m\in \mathbb N$ such that
$$
a^mba=a^m, \hskip.5truecm  bab=b,\hskip.5truecm ab=ba.
$$
\indent If the Drazin inverse of $a$ exists, then it will be denoted by $a^D$, 
see for example \cite{Dr, BS, B}. \par

\indent On the other hand, the notion of \it regularity \rm was introduced in \cite{KM}. 
 Recall that,
given a unital Banach algebra $A$ and a regularity $\mathcal R\subseteq A$,
the \it spectrum derived from the 
regularity \rm $ \mathcal R$  is defined by $\sigma_{\mathcal R} (a)=\{ \lambda\in \mathbb C\colon a-\lambda\notin  \mathcal R\}$,
where $a\in A$ and $ a-\lambda$ stands for  $a-\lambda e$ (see \cite{KM}).  Next consider the set
$\mathcal{DR} (A)$ = $\{ a\in A\colon  \hbox{  a is Drazin invertible}\}$. According to
\cite[Theorem 2.3]{BS}, $\mathcal{DR} (A)$ is a regularity. This fact led to the following definition
  \cite{BS}.\par
\begin{df} Let $A$ be a unital Banach algebra. The Drazin spectrum
of an element $a\in A$ is the set
$$
\sigma_{ \mathcal{ DR} }(a) = \{\lambda\in \mathbb C \colon a-\lambda\notin  \mathcal{ DR}(A)\}.
$$
\end{df}

\indent Note that $\sigma_{ \mathcal{ DR} }(a)\subseteq \sigma(a)$, the spectrum of $a$, and according to \cite[Theorem 1.4]{KM}, the Drazin spectrum
of a Banach algebra element satisfies the spectral mapping theorem for analytic functions
defined on a neighbourhood of the ususal spectrum which are non-constant on each component of its 
domain of definitioin. In addition, according to \cite[Proporsition 2.5]{BS},
$\sigma_{ \mathcal{ DR} }(a)$ is a closed subset of $\mathbb C$. \par

\indent Recall that according to \cite[Theorem 12]{B},
$$
\sigma(a)=\sigma_{\mathcal{DR}}(a)\cup \Pi(a),\hskip.3truecm 
\sigma_{\mathcal{DR}}(a)\cap \Pi(a)=\emptyset,\hskip.3truecm
\sigma_{\mathcal{DR}}(a) =\hbox{\rm acc }\sigma (a)\cup \mathcal{IES}(a)
$$
where $\Pi(a)$ is the set of poles of the resolvent of $a$ (see \cite[Remark 10]{B}), $\mathcal{IES}(a)
=$ iso $\sigma(a)\setminus\Pi (a)$, and if $K\subseteq \mathbb C$,
then iso $K$ denote the set of isolated points of $K$ and acc $K=K\setminus$ iso $K$.\par

\indent When $X$ is a Banach space,  $A=L(X)$ will denote the
Banach algebra of all operators defined on and with values in $X$. 
In addition, if $T\in L(X)$,  then $N(T)$ and $R(T)$ will stand for the null space and the
range of $T$ respectively. Recall that the \it descent \rm and the \it ascent \rm of $T\in L(X)$ are
$d(T) =\hbox{ inf}\{ n\ge 0\colon  R(T^n)=R(T^{n+1})\}$ and
$ a(T)=\hbox{ inf}\{ n\ge 0\colon  N(T^n)=N(T^{n+1})\}$
respectively, where if some of the above sets is empty, its infimum is then defined as $\infty$, see  
for example \cite{K, MM, BKMO, BBO,B}.\par 

\indent Given a Banach space $X$, the left and the right Drazin spectra
of an oper-
ator have been introduced. The mentioned spectra are the ones derived from the regularities
$\mathcal{ LD}(X)=\{ T\in L(X)\colon a(T) \hbox{ is finite and  }
R(T^{ a(T) +1 })\hbox{ is closed}\}$ and
$\mathcal{ RD}(X) = \{ T\in L(X)\colon d(T) \hbox{  is finite and  }$
$R(T^{ d(T) })\hbox{ is closed}\}$ respectively,
and they will be denoted by
$\sigma_{  \mathcal{ LD} }(T)$ and $\sigma_{  \mathcal{ RD} }(T)$ respectively. In addition, concerning the basic 
properties of the left and the right Drazin spectra, see \cite{MM, BBO}. \par

\indent Furthermore, if $X$ is a Banach space, then according again to \cite{MM}, the sets
$\mathcal R^a_4(X) = \{ T\in L(X)\colon d(T)\hbox{  is finite} \}$ and
 $\mathcal R^a_9(X) = \{ T\in L(X)\colon a(T) \hbox{  is finite}\}$ are regularities with well-defined  spectra derived from them, namely the \it descent spectrum \rm and the \it ascent spectrum \rm 
respectively. The descent spectrum was also studied in \cite{BKMO} and it will be denoted by
 $\sigma_{dsc}(T)$, $T\in L(X)$. On the other hand, $\sigma_{asc}(T)$ will stand for the ascent spectrum.
\par

\section{Properties of the Drazin spectrum} \label{sect2}

\indent Given a unital Banach algebra $A$ and $a\in A$, 
according to \cite[Proposition 1.5]{L}, necessary and sufficient
for the Drazin spectrum of $a$ to be empty is the fact that
the usual spectrum of $a$ consists in a finite number of poles of the resolvent of $a$. 
In the following theorem several other equivalent conditions 
of this property will be presented, compare with \cite[Theorem 1.5]{BKMO} and \cite[Theorem 2.7]{BBO}. 
To this end, recall first that an
element $a\in A$ is said to be \it algebraic\rm, if there exists a polynomial
$P$ with coeficients in $\mathbb C$ such that $P(a)=0$. In addition,
$\partial \sigma (a)$ will stand for the topological boundary of
$\sigma (a)$ and $\rho_{\mathcal{DR}}(a)$ for the \it Drazin resolvent set of $a$\rm,
i.e., $\rho_{\mathcal{DR}}(a)=\mathbb C\setminus \sigma_{\mathcal{DR}}(a)$.\par 

\begin{thm}Let $A$ be a unital Banach algebra and consider $a\in A$. 
Then, the following statements are equivalent.
$$
(i)\hskip.2truecm \sigma_{\mathcal{DR}}(a)=\emptyset,  
\hskip1cm (ii) \hskip.2cm \partial\sigma(a)\subseteq \rho_{\mathcal{DR}}(a),
\hskip1cm (iii)\hskip.2cm a \hbox{ is algebraic}.
$$
\end{thm}
\begin{proof} Note that  $\sigma_{\mathcal{DR}}(a)=\emptyset$ if and only if $\rho_{\mathcal{DR}}(a)=\mathbb C$. Then,
it is clear that the first statement implies the second.\par
\indent Suppose that $\partial\sigma(a)\subseteq \rho_{\mathcal{DR}}(a)$. Then according to \cite[Theorem 12]{B},
$$
\Pi (a)\subseteq \hbox{\rm iso }\sigma(a)\subseteq \partial\sigma(a)\subseteq \rho_{\mathcal{DR}}(a)\cap \sigma (a)=\Pi(a).
$$
Consequently,  $\Pi (a)=\hbox{\rm iso }\sigma(a)= \partial \sigma(a)$. What is more, since 
$\hbox{\rm acc }\sigma (a)=\sigma (a)\setminus \hbox{\rm iso }\sigma (a)=\sigma (a)\setminus\partial \sigma (a)$,
$\hbox{\rm acc }\sigma (a)$ coincides with the interior set of $\sigma (a)$. Thus, since $\hbox{\rm acc }\sigma (a)$ is closed and
open, $\hbox{\rm acc }\sigma (a)=\emptyset$. In particular, according to \cite[Theorem 12(iv)]{B},   $\sigma_{\mathcal{DR}}(a)$
is empty.\par
\indent To prove the equivalence between statements (i) and (iii), in first place the case $A=L(X)$, $X$ a Banach space, will be considered.\par
\indent Let $T\in L(X)$ be an algebraic operator and consider $P\in \mathbb C[X]$ the minimal polynomial such that $P(T)=0$. Then,
from this condition and the spectral mapping theorem follows that $\sigma (T)=\{ \lambda\in \mathbb C\colon P(\lambda)=0\}$.
In particular, if $\sigma (T)=\{ \lambda_i\colon 1\le i\le n\}$, then according to \cite[Theorem 5.7-A]{T}
and \cite[The-
orem 5.7-B]{T}, there exist $(M_i)_{1\le i\le n }$
such that $M_i\subseteq X$ is an invariant closed subspace for $T$, $1\le i\le n$,  $X=\oplus_{1\le i\le n }M_i$, and
if $T_i=T\mid_{M_i }^{M_i }$, then $\sigma(T_i)=\{\lambda_i\}$. Now well, given a fixed $i$, $1\le i\le n$, since $\lambda_j\notin\sigma (T_i)$, $1\le j\le n$, 
$j\neq i$, and $P(X)=\Pi_{k=1}^n (X-\lambda_k)^{n_k}$ for some $n_k\in \mathbb N$, $1\le k\le n$, $(T_i-\lambda_i)^{n_i}=0$. Therefore, $T_i-\lambda$ is Drazin invertible in $M_i$ for all $\lambda\in\mathbb C$, $1\le i\le n$. Since $X=\oplus_{1\le i\le n }M_i$
and $T-\lambda=\oplus_{i=1}^n (T_i-\lambda)$, 
$T-\lambda$ is Drazin invertible for all $\lambda\in \mathbb C$.\par

\indent On the other hand, if $\sigma_{\mathcal{DR}}(T)=\emptyset$, then according to \cite[Theorem 12(i)]{B},
$\sigma (T)=\Pi (T)$, actually $\sigma (T)$  is a finite set of poles. If 
$\sigma (T)=\{\lambda_i\colon 1\le i\le n\}$, then as in the previous paragraph consider $(M_i)_{1\le i\le n }$
such that $M_i\subseteq X$ is an invariant closed subspace for $T$, $1\le i\le n$,  $X=\oplus_{1\le i\le n }M_i$, and
if $T_i=T\mid_{M_i }^{M_i }$, then $\sigma(T_i)=\{\lambda_i\}$. Moreover, according to \cite[Theorem 5.8-A]{T},
a straightforward calculation proves that $\lambda_i$ is a pole of $T_i$. Then, 
$M_i=N(T_i-\lambda_i)^{m_i }\oplus R(T-\lambda_i)^{m_i }$ for some $m_i\in \mathbb N$, what is more,
$T_i-\lambda_i$ restricted to $R(T-\lambda_i)^{m_i }$ is invertible, see \cite[Theorem 5.8-A]{T}.  Now well, given $i$ such that
$1\le i\le n$, since $N_i=N(T_i-\lambda_i)^{m_i }$
and $R_i=R(T-\lambda_i)^{m_i }$ are closed invariant subspaces for $T_i-\lambda_i$, 
if $T_i^1$ and $T_i^2$ are the restrictions of $T_i$ to $N_i$ and $R_i$ respectively,
then $\sigma (T_i)
=\sigma (T_i^1)\cup\sigma (T_i^2)$, see \cite[Theorem 5.4-C]{T}. However, since $\sigma(T_i)=\{\lambda_i\}$, $N_i\neq 0$ and  $\lambda_i\in \sigma (T_i^1)\setminus \sigma (T_i^2)$, 
$\sigma(T_i^2)=\emptyset$. Therefore $(T_i-\lambda_i)^{m_i}=0$. Define $P(X)=\Pi_{i=1}^n (X-\lambda_i)^{m_i}$.
Since $X=\oplus_{1\le i\le n }M_i$, it is not difficult to prove that $P(T)=0$.\par

\indent To conclude the proof, use \cite[Theorem 4(iv)]{B} and the fact that $a\in A$ is algebraic if and
only if $L_a\in L(A)$ is algebraic.
\end{proof}

\indent Next at most countable Drazin spectrum will be characterized, see also \cite[Corollary 1.8]{BKMO} and \cite[Corollary 2.10]{BBO}.\par

\begin{thm}Let $A$ be a unital Banach algebra and consider $a\in A$. Then, the following
statements are equivalent.
$$
(i)\hskip.2cm \sigma(a) \hbox{   is at most countable},\hskip1cm
(ii)\hskip.2cm \sigma_{\mathcal{DR}}(a) \hbox{   is at most countable}.
$$
\indent Furthermore, in this case 
\begin{align*}
\sigma_{\mathcal{DR}}(a) &=\sigma_{\mathcal{LD}}(L_a) =\sigma_{\mathcal{RD}}(L_a)=\sigma_{dsc}(L_a)&\\
&=\sigma_{\mathcal{LD}}(R_a) =\sigma_{\mathcal{RD}}(R_a)=\sigma_{dsc}(R_a).&\\
\end{align*}
\end{thm} 
\begin{proof} Clearly, the first statement implies the second.
On the other hand, since according to \cite[Theorem 12(i)]{B},
$\sigma(a)=\sigma_{\mathcal{DR}}(a)\cup \Pi (a)$, being $\Pi (a)$
a countable set, the second statement implies the first.\par
\indent Concerning the remaining identities, since $\sigma (a)=\sigma (L_a)$,
\cite[Chapter I, section 5, Proposition 4]{BD},
according to \cite[Corollary 1.8]{BKMO} and \cite[Corollary 2.10]{BBO},
$$
\sigma (L_a)=\sigma_{dsc}(L_a)\cup \Pi (L_a)=\sigma_{\mathcal{LD}}(L_a)\cup \Pi(L_a).
$$
 Moreover, according to \cite[Theorem 1.5]{BKMO} and \cite[Theorem 2.7]{BBO},
$\sigma_{dsc}(L_a)\cap \Pi (L_a)=\emptyset=\sigma_{\mathcal{LD}}(L_a)\cap \Pi(L_a)$.
Therefore, according  to  \cite[Theorem 12]{B} and to the fact that
$\sigma_{\mathcal{LD}}(L_a)\subseteq \sigma_{\mathcal{DR}}(L_a)$ and 
$\sigma_{dsc}(L_a)\subseteq \sigma_{\mathcal{RD}}(L_a)\subseteq \sigma_{\mathcal{DR}}(L_a)$ (\cite[Theorem 3]{B}),
$$
\sigma_{\mathcal{DR}}(L_a)=\sigma_{\mathcal{LD}}(L_a) =\sigma_{\mathcal{RD}}(L_a)=\sigma_{dsc}(L_a).
$$
Since, according to \cite[Theorem 4(iv)]{B}, $\sigma_{\mathcal{DR}}(a)=\sigma_{\mathcal{DR}}(L_a)$,
the first identity holds. To prove the identity involving $R_a$, interchange $L_a$ with $R_a$
and use the same argument.
\end{proof}

\indent Given a Banach algebra $A$ and $a$ and $b\in A$, the identity
$$
\sigma(ab)\setminus\{ 0\}=  \sigma(ba)\setminus\{ 0\},
$$
is well known (\cite[Chapter I, section 5, Proposition 3]{BD}).
However, in the case of the Drazin spectrum, both spectra coincide,
see also \cite[Theorem 2.11]{BBO}.\par
 
\begin{thm}\label{thm3} Let $A$ be a unital Banach algebra and consider $a$ and $b\in A$. Then,
$$
\sigma_{\mathcal{DR}}(ab)=\sigma_{\mathcal{DR}}(ba).  
$$
\end{thm}
\begin{proof} First of all the Banach space operator case will be studied. \par

\indent Let $X$ be a Banach space and consider $S$ and $T\in L(X)$. Then, according to 
\cite[Theorem 3]{B} and \cite[Theorem 1]{BF},
$$ 
\sigma_{\mathcal{DR}}(ST)\setminus\{ 0\}=\sigma_{\mathcal{DR}}(TS)\setminus\{ 0\}.  
$$
\indent Therefore, in order to conclude that the Drazin spectra of $ST$ and $TS$ coincide,
it is enough to prove that $ST$ is Drazin invertible if and only if $TS$ is Drazin invertible.
However, this statement is a consequence of the equivalence between \cite[Proposition 2.1(ii)]{R}
and \cite[Proposition 2.1(iii)]{R}.\par
\indent Finally, the general case can be proved applying \cite[Theorem 4(iv)]{B}.
\end{proof} 

\begin{rema}\rm  Let $A$ be a unital Banach algebra.
Note that according to Theorem \ref{thm3}, \cite[Theorem 12(i)]{B} and the identity $\sigma(ab)\setminus\{ 0\}=  \sigma(ba)\setminus\{ 0\}$,
$$
\Pi(ab)\setminus\{ 0\}=\Pi(ba)\setminus\{ 0\},
$$
where $a$ and $b$ belong to $A$. Furthermore, this identity can not be improved.
Consider for example the Banach space $X=l^2 (\mathbb N)$ 
and the operator
$S$, $T\in L(X)$  defined by
$$
S((x_n)_{n\ge 1})=(0,x_1,x_2, \ldots, x_n,\ldots),\hskip.5truecm T((x_n)_{n\ge 1})=(x_2, x_3,\ldots, x_n,\ldots ),
$$
where $(x_n)_{n\ge 1}\in X$. Then, a straightforward calculation proves that 
$$
\Pi (ST)=\{0,1\},\hskip1truecm \Pi (TS)=\{1\}.
$$
\end{rema}

\indent Next the relationships among the Drazin spectra of the multiplication operators and the adjoint of a $C^*$-algebra
will be studied. First of all the operator case will be considered.\par 

\begin{rema} \rm Given $X$ a Banach space and $T\in L(X)$, note that $\sigma_{\mathcal{DR}}(T^*)=\sigma_{\mathcal{DR}}(T)$,
where $T^*\in L(X^*)$ is the adjoint of $T$ and $X^*$ is the dual space of $X$.
Consequently, according to \cite[Theorem 4(iv)]{B}, all the sets $\sigma_{\mathcal{DR}}(T)$, $\sigma_{\mathcal{DR}}(T^*)$,
$\sigma_{\mathcal{DR}}(L_T)$, $\sigma_{\mathcal{DR}}(L_{T^*})$, $\sigma_{\mathcal{DR}}(R_T)$, and $\sigma_{\mathcal{DR}}(R_{T^*})$
coincide.\par

\indent In addition, according to \cite[page 139]{MM}, $\sigma_{\mathcal{LD}}(T^*)=\sigma_{\mathcal{RD}}(T)$ and 
$\sigma_{\mathcal{RD}}(T^*)=\sigma_{\mathcal{LD}}(T)$.
However, when $X$ is a Hilbert space, according to \cite[page 139]{MM} and \cite[Theorem 9]{B}
\begin{align*} 
&(i)&&\sigma_{  \mathcal{RD}}(T)= \sigma_{ \mathcal{RD}}(L_T)=  \sigma_{\mathcal{LD}}(R_T)=\overline{\sigma_{  \mathcal{LD}}(T^*)}=
\overline{\sigma_{ \mathcal{LD}}(L_{T^*})}.&\\
&\hbox{ }&&\hskip1.36truecm=\overline{\sigma_{ \mathcal{RD}}(R_{T^*})}.&\\
&(ii)&&\sigma_{  \mathcal{LD}}(T)= \sigma_{ \mathcal{LD}}(L_T)=  \sigma_{\mathcal{RD}}(R_T)=\overline{\sigma_{  \mathcal{RD}}(T^*)}=
\overline{\sigma_{ \mathcal{RD}}(L_{T^*})}.&\\
&\hbox{ }&&\hskip1.36truecm=\overline{\sigma_{ \mathcal{LD}}(R_{T^*})},&\\
\end{align*} 
where if $\lambda\in \mathbb C$, then $\overline{\lambda}$ denotes the complex conjugate of $\lambda$, and
if $A\subseteq\mathbb C$, then $\overline{A}=\{\overline{\alpha}\colon \alpha\in A\}$.\par
\indent What is more, if $T=T^*$, then a straightforward calculation proves that the sets considered in (i) and (ii) are
contained in the real line and they all coincide. In the following theorem, similar identities will be proved for $C^*$-algebra
elemets. However, before going on three facts should be mentioned.\par
\indent First, given a Banach algebra $A$, recall that $a\in A$ is said to be \it regular\rm, if there exists $b\in A$ such that $a=aba$, see \cite{HM}. 
Second, given a $C^*$-algebra $A$ and $a\in A$, note that $ \sigma_{\mathcal{DR}}(a^*)=\overline{\sigma_{\mathcal{DR}}}(a)$. Third, 
some notation. Given a $C^*$-algebra $A$, if $B\subseteq A$, then $B^*=\{\beta^*\colon \beta\in B\}$.
\end{rema} 

\begin{thm}\label{thm10} Let $A$ be a $\Bbb C^*$-algebra. Then, the following statements hold.\par
\begin{align*}
 &(i)& &\sigma_{dsc}(L_{a^*})=\overline{\sigma_{dsc}(R_a)},&&\sigma_{dsc}(R_{a^*})=\overline{\sigma_{dsc}(L_a)}.&\\
&(ii)& &\sigma_{\mathcal{RD}}(L_{a^*})=\overline{\sigma_{\mathcal{RD}}(R_a)},&&\sigma_{\mathcal{RD}}(R_{a^*})=\overline{\sigma_{\mathcal{RD}}(L_a)}.&\\
&(iii)& &\sigma_{asc}(L_{a^*})=\overline{\sigma_{asc}(R_a)},&&\sigma_{asc}(R_{a^*})=\overline{\sigma_{asc}(L_a)}.&\\
\end{align*}
\begin{align*}
&(iv)& &\sigma_{\mathcal{LD}}(L_{a^*})=\overline{\sigma_{\mathcal{LD}}(R_a)},&&\sigma_{\mathcal{LD}}(R_{a^*})=\overline{\sigma_{\mathcal{LD}}(L_a)}.&\\
\end{align*}

\indent Furthermore, when $a$ is a hermitian element of $A$, all the spectra considered in
statements (i)-(iv) are contained in the real line, and
\begin{align*}
 &(v)&&\sigma_{asc}(L_a)=\sigma_{asc}(R_a),\hskip.1truecm\sigma_{\mathcal{LD}}(L_a)=\sigma_{\mathcal{LD}}(R_a),\\
 &(vi)& &\sigma_{\mathcal{DR}}(a)=\sigma_{dsc}(L_a)=\sigma_{dsc}(R_a)=\sigma_{\mathcal{RD}}(L_a)=\sigma_{\mathcal{RD}}(R_a).&\\
\end{align*}
\end{thm}
\begin{proof} Let $a\in A$. According to the fact that $(R(L_a))^*=R(R_{a^*})$, a straightforward
calculation proves that necessary and sufficient for $d(L_a)$ to be finite is the fact
that $d(R_{a^*})$ is finite. Moreover, in this case $d(L_a)=d (R_{a^*})$. However, 
from this identity the first statement can be easily deduced.\par

\indent On the other hand, due to the fact that $b\in A$ is a regular element of $A$ if and only if
$b^*$ is regular, if $0\notin \sigma_{dsc}(L_a)$, $d=d(L_a)$, and $R(L_{a^d})$
is closed, then according to \cite[Theorem 2]{HM}, \cite[Theorem 8]{HM}, and what has been proved,
$0\notin \sigma_{dsc}(R_{a^*})$, $d=d(R_{a^*})$, and $R(R_{a*^d})$ is closed. Therefore,
$\sigma_{\mathcal{RD}}(R_{a^*})\subseteq\overline{\sigma_{\mathcal{RD}}(L_a)}$. A similar argument proves the
other inclusion. Interchanging $a$ with $a^*$, the remaining identity can be proved.\par

\indent According to the identity $(N(L_a))^*=N(R_{a^*})$, it is not difficult to prove that
necessary and sufficient for $a(L_a)$ to be finite is the fact that $a(R_{a^*})$
is finite. Futhermore, in this case $a(L_a)=a(R_{a^*})$. As in the case of
the descent spectrum, from this identity the third statement can be proved.\par

\indent In order to prove the fourth statement, apply an argument similar to the one used to prove the second identity of statement (ii),
considering the ascent spectrum instead of the descent spectrum.\par

\indent As regard the last statements, if $a=a^*$, then an easy calculation proves that
$\sigma_{\mathcal{DR}}(a)$ is contained in the real line. 
However, according to \cite[Theorem 3]{B} and \cite[Theorem 4(iv)]{B}, all the spectra in statements (i)-(iv) are contained  in $\mathbb R$. Moreover,
$\sigma_{dsc}(L_a)=\sigma_{dsc}(R_a)$, $\sigma_{\mathcal{RD}}(L_a)=\sigma_{\mathcal{RD}}(R_a)$,
$\sigma_{asc}(L_a)=\sigma_{asc}(R_a)$, $\sigma_{\mathcal{LD}}(L_a)=\sigma_{\mathcal{LD}}(R_a)$. Finally, statement (vi) can be easily derived from \cite[Theorem 5(ii)]{B} and
what has been proved. 
\end{proof}   

\indent Recall that, according to the example \cite[page 265]{BKMO}, there is no relation that lies 
the descent of $R_T$ to the descent of $T$, $T\in L(H)$, $H$ a Hilbert space. 
In the following theorem the ascent and the descent spectra of $R_T$ will be characterized.
See also \cite[Theorem 8]{B} and \cite[Theorem 9]{B}.\par

\begin{thm}
Let $H$ be a Hilbert space. Then, the following statement hold.\par
\begin{align*}
&(i)& & \sigma_{dsc}(R_T)=   \overline{\sigma_{dsc}(L_{T^*})}=  \overline{\sigma_{dsc}(T^*)},&
 & \sigma_{dsc}(R_{T^*})=   \overline{\sigma_{dsc}(L_T)}=  \overline{\sigma_{dsc}(T)}.&\\
&(ii)& & \sigma_{asc}(R_T)=   \overline{\sigma_{asc}(L_{T^*})}=  \overline{\sigma_{asc}(T^*)},&
& \sigma_{asc}(R_{T^*})=   \overline{\sigma_{asc}(L_T)}=  \overline{\sigma_{asc}(T)}.&\\
\end{align*} 

\indent Furthermore, when $T=T^*$, $\sigma_{dsc}(R_T)$ and  $\sigma_{asc}(R_T)$ are subsets of the real line, and 
\begin{align*}
&(iii)& & \sigma_{asc}(R_T)=  \sigma_{asc}(L_T)= \sigma_{asc}(T),&\\
&(iv)&  &\sigma_{\mathcal{DR}}(T)=\sigma_{dsc}(R_T)=\sigma_{dsc}(L_T)=\sigma_{dsc}(T).&\\
\end{align*} 
\end{thm}
\begin{proof} Apply Theorem \ref{thm10}, \cite[Theorem 8]{B} and \cite[Theorem 9]{B}.
\end{proof}

\section{Meromorphic Banach space operators} \label{sect3}

\indent Recall that given $X$ a Banach space and $T\in L(X)$, the operator $T$ is said to be  \it meromorphic\rm, 
if $\sigma (T)\setminus \{0\}\subseteq  \Pi (T)$. As an application of the properties of the Drazin spectrum, 
in the following theorem several results on
 meromorphic Banach space operators will be proved, see \cite[Corolalry 1.9]{BKMO}
and the paragraph that follows \cite[Corollary 2.10]{BBO}.\par

\begin{thm} Let $X$ be a Banach space and consider $T\in L(X)$. \par
\noindent (a) The following statements are equivalent.
\begin{align*}
& (i)\hskip.2truecmT\hbox{ is meromorphic},& &(ii)\hskip.2truecm \sigma_{\mathcal{DR}}(T)\subseteq \{0\},&\\
& (iii)\hskip.2truecm L_T\in L(L(X))\hbox{ is meromorphic},&
&(iv)\hskip.2truecm R_T\in L(L(X))\hbox{ is meromorphic}.&\\
\end{align*}
\noindent (b) Let $S$ and $T\in L(X)$ . Then, necessary and sufficient for $ST$ to be meromorphic is the fact that
$TS$ is meromorphic.\par
\noindent (c) Let $F\in L(X)$ and suppose that there exists a positive integer $n$ such that $F^n$ has finite dimensional range 
and $F$ commutes with $T$. Then,
if $T$ is meromorphic, $T+F$ is meromorphic.
\end{thm}
\begin{proof} To prove the equivalence between statements  a(i) and a(ii), apply \cite[Theorem 12]{B}.
Next, to prove the equivalence between statements a(ii)-a(iv), apply \cite[Theorem 4(iv)]{B}.\par
\indent To prove the second statement, apply Theorem \ref{thm3} and statement (a)(ii).\par
\indent Concerning the third statement, given $U\in L(X)$ and $F$ as in statement (c), according to  \cite[Theorem 4]{K} and \cite[Theorem 2.2]{KL},
$U$ is Drazin invertible if and only if $U+F$ is Drazin invertible.   
In particular, $\sigma_{\mathcal{DR}}(U)=\sigma_{\mathcal{DR}}(U+F)$. Clearly,
statement (c) follows from this identity and a(ii).
\end{proof}

\vskip.3truecm
\noindent Enrico Boasso\par
\noindent E-mail address: enrico\_odisseo@yahoo.it 
\end{document}